\documentclass[11pt]{amsart}

\usepackage{graphicx} 

\usepackage{color}

\usepackage[all]{xypic} 

\usepackage{times} 

\usepackage[latin1]{inputenc} 

\usepackage[english,brazil]{babel}

\usepackage{hyperref} 

\theoremstyle{plain}

\newtheorem{thm}{Theorem}[section]
\newtheorem{lem}[thm]{Lemma}

\newtheorem{con}[thm]{Conjecture}
\newtheorem{cons}[thm]{Consequence}
\swapnumbers

\title{Some consequences of the Firoozbakht's conjecture}
\author[{Luan Alberto Ferreira and Hugo Luiz Mariano}]{Luan Alberto Ferreira\\ \\ Hugo Luiz Mariano}
\address{IME - USP\hfill\break\indent Rua do Matão, 1010\hfill\break\indent  São Paulo - SP, 05508-090\hfill\break\indent  Brazil}
\email{luan@ime.usp.br \\ hugomar@ime.usp.br}
\date{February 2017}
\thanks{The first author were supported by CAPES-PROEX}
\begin{document}
\selectlanguage{english} 

\setcounter{page}{1} 

\maketitle
\begin{abstract}
In this paper we present the statement of the Firoozbakht's conjecture, some of its consequences if it is proved and we show a consequence of Zhang's theorem concerning the Firoozbakht's conjecture.

%
\end{abstract}

\newcommand{\fantasma}[3]{{#1}_{{#2}}^{\phantom{{#2}}{#3}}}


\tableofcontents

\section{Statement of the conjecture}

What makes a math problem a good one? For the authors of this note, the answer to this question involves three major ingredients:

\begin{itemize}
  \item It's easy to understand;
  \item It's very difficult to solve;
  \item It has several strong implications.
\end{itemize}

An example of a problem satisfying the three above conditions is the Firoozbakht's conjecture (see \cite{Wikipedia}). In 1982 (see \cite{site1}), the Iranian mathematician Farideh Firoozbakht, from the University of Isfahan, conjectured the following:

\vspace{3mm}

\noindent \textbf{Firoozbakht's conjecture:} \emph{Let $\{p_n\}_{n \in \mathbb{N}}$ the sequence of prime numbers. Then the sequence $\{\sqrt[n]{p_n}\}_{n \in \mathbb{N}}$ is strictly decreasing.}

\vspace{3mm}

Let's see the first cases:

\begin{center}
\begin{tabular}{ccc}
$n$ & $p_n$ & $\sqrt[n]{p_n}$ \\
& & \\
1 & 2 & 2.0000 \\
2 & 3 & 1.7321 \\
3 & 5 & 1.7100 \\
4 & 7 & 1.6266 \\
5 & 11 & 1.6154 \\
6 & 13 & 1.5334 \\
7 & 17 & 1.4989 \\
8 & 19 & 1.4449 \\
9 & 23 & 1.4168 \\
10 & 29 & 1.4004 \\
\end{tabular}
\end{center}

The Firoozbakht's conjecture has been verified for all primes below $4 \cdot 10^{18}$ (see \cite{K1}), but there is no consensus about the truth of this conjecture. The authors of this paper, in particular, believes it is true. The importance of the Firoozbakht's conjecture is that this affirmation is the boldest (and reasonable) known statement about prime gaps. The $n$-th prime gap $g_n$ is defined by the relation $$g_n = p_{n+1} - p_n, \ \forall \ n \in \mathbb{N}.$$

Utilizing this language, the Firoozbakht's conjecture can be written in following way:

\vspace{3mm}

\noindent \textbf{Firoozbakht's conjecture, prime gaps version:} $g_n \leq p_n(\sqrt[n]{p_n} - 1), \ \forall \ n \in \mathbb{N}$.

\vspace{3mm}

Is this paper, we will present a very strong consequence of the Firoozbakht's conjecture concerning prime gaps and explain how this can solve a lot of other problems about prime numbers. We also show that Zhang's theorem implies Firoozbakht's conjecture is true for infinitely many values of $n$.

\section{Relationship with prime gaps}

The next lemma will be used further.

\begin{lem}

$\displaystyle \frac{g_n}{g_n + p_n} < \ln(p_{n + 1}) - \ln(p_n), \ \forall \ n \in \mathbb{N}$.

\end{lem}

\begin{proof} We know by calculus that $1 + \ln(x) < x, \ \forall \ x \in (0, 1)$. Taking $x = \displaystyle \frac{p_n}{p_{n + 1}}$, we obtain the result. \end{proof}

The major consequence of the Firoozbakht's conjecture is the following inequality:

\begin{thm} \label{firoozbakhtlnquadrado}

If the Firoozbakht's conjecture is true, then $$g_n < \ln^2(p_n) - \ln(p_n) - 1, \ \forall \ n \geq 10.$$ In particular, $$g_n < \ln^2(p_n) - \ln(p_n), \ \forall \ n \geq 5,$$ and $$\displaystyle \limsup_{n \rightarrow \infty} \displaystyle \frac{g_n}{\ln^2(p_n)} \leq 1.$$

\end{thm}

\begin{proof} (Following \cite{K2}) Suppose that $$\sqrt[n]{p_n} > \sqrt[n+1]{p_{n+1}}, \ \forall \ n \in \mathbb{N}.$$ Taking the natural logarithm, we obtain $$\displaystyle \frac{\ln(p_n)}{n} > \displaystyle \frac{\ln(p_{n+1})}{n+1}, \ \forall \ n \in \mathbb{N}.$$ Isolating $n$, we obtain $$n < \displaystyle \frac{\ln(p_n)}{\ln(p_{n + 1}) - \ln(p_n)}, \ \forall \ n \in \mathbb{N}.$$ Utilizing the previous lemma, we have $$n < \ln(p_n) + \displaystyle \frac{\ln(p_n) \cdot p_n}{g_n}, \ \forall \ n \in \mathbb{N}.$$ On the other hand, a calculation shows that $$\displaystyle \frac{x+\ln^2(x)}{\ln(x)-1-\frac{1}{\ln(x)}} < \displaystyle \frac{x}{\ln(x)-1-\frac{1}{\ln(x)}-\frac{1}{\ln^2(x)}}, \ \forall \ x \geq 285967.$$ If $\pi(x)$ denotes the prime-counting function, then \cite[Corollary 3.6]{A} states that $$\displaystyle \frac{x}{\ln(x)-1-\frac{1}{\ln(x)}-\frac{1}{\ln^2(x)}} < \pi(x), \ \forall \ x \geq 1772201.$$ Then $$\displaystyle \frac{x + \ln^2(x)}{\ln(x) - 1 - \frac{1}{\ln(x)}} < \pi(x), \ \forall \ x \ge 1772201.$$ Taking $x = p_n$ and remembering that $1772201 = p_{133115}$, we get $$\displaystyle \frac{p_n + \ln^2(p_n)}{\ln(p_n) - 1 - \frac{1}{\ln(p_n)}} < \pi(p_n) = n < \ln(p_n) + \displaystyle \frac{\ln(p_n) \cdot p_n}{g_n}, \ \forall \ n \ge 133115.$$ Isolating $g_n$, we get $$g_n < \displaystyle \frac{p_n}{p_n + \ln(p_n) + 1} \cdot [\ln^2(p_n) - \ln(p_n) - 1], \ \forall \ n \geq 133115.$$ But $$\displaystyle \frac{p_n}{p_n + \ln(p_n) + 1} < 1, \ \forall \ n \in \mathbb{N}.$$ Then $$g_n < \ln^2(p_n) - \ln(p_n) - 1, \ \forall \ n \geq 133115.$$ By the help of a computer for checking the remaining cases ($10 \leq  n \leq 133114$),  we finish the proof. \end{proof}

The combination of the previous theorem with the following result provides another tight relation between prime gaps and Firoozbakht's conjecture.

\begin{thm}

If $g_n < \ln^2(p_n) - \ln(p_n) - 1,17, \ \forall \ n \geq 10$, then the Firoozbakht's conjecture is true.

\end{thm}

For a proof of this result we refer the reader to \cite{K2}. Now we will see how theorem \ref{firoozbakhtlnquadrado} can be used to solve long standing classical problems concerning prime numbers.

\section{Other consequences}

The best upper bound on prime gaps currently known is due to Baker, Harman and Pintz:

\begin{thm}[Baker-Harman-Pintz]

$g_n \le p_n^{0.525}, \ \forall \ n \gg 0$.

\end{thm}

It is easy to see that Firoozbakht's conjecture improves the Baker-Harman-Pintz's bound significantly. For a proof of the Baker-Harman-Pintz's theorem we refer the reader to \cite{BHP}.

For the next consequence we will use the following lemma.

\begin{lem}

Is Firoozbakht's conjecture is true, then $g_n < \sqrt{n}, \ \forall \ n \geq 3645$.

\end{lem}

\begin{proof} A simple consequence of Prime Number Theorem is that $$p_n < n^2, \ \forall \ n \geq 2.$$ So, if Firoozbakht's conjecture is true, then $$g_n < \ln^2(p_n) - \ln(p_n) - 1 < \ln^2(n^2) - \ln(n^2) - 1, \ \forall \ n \geq 10.$$ By calculus, we know that $$\ln^2(n^2) - \ln(n^2) -1 \leq \sqrt{n}, \ \forall \ n \geq 411781.$$ So $$g_n < \sqrt{n}, \ \forall \ n \geq 411781.$$  Verifying the remaining cases ($3645 \leq n \leq 411780$) by a help of a computer, we finish the proof. \end{proof}

\begin{cons} [Sierpinski]

Let $n$ be an integer greater then $1$. If you write the numbers $1, 2, \ldots, n^2$ in a matrix in the following way $$\begin{array}{cccc} 1 & 2 & \ldots & n \\ n+1 & n+2 & \ldots & 2n \\ 2n+2 & 2n+2 & \ldots & 3n \\ \ldots & \ldots & \ldots & \ldots \\ (n-1)n+1 & (n-1)n+2 & \ldots & n^2\end{array}$$ then each row contains, at least, a prime number. Moreover, for each positive integer $k$ exists another integer $n_0(k)$ such that if $n \geq n_0(k)$, then each row contains, at least, $k$ prime numbers.

\end{cons}

\begin{proof} Let $n$ be an integer, $n \geq 34123$. As $34123 = p_{3645}$ is in the first row, suppose by absurd in the $k$-th row, $2 \leq k \leq n$, there is no prime number. So $(k-1)n+1, (k-1)n+2, \ldots, kn$ are all composite numbers. Let $p_m$ be the largest prime less then $(k-1)n+1$. Then $g_m \geq n$. But $m < p_m \leq n^2$, which implies $\sqrt{m} < n$. Also, as $p_m \geq p_{3645}$, then $m \geq 3645$. By the previous lemma, $g_m < \sqrt{m} < n$, absurd. By checking the remaining cases ($2 \leq n \leq 34122$) by a help of a computer, we finish the proof of the first statement.

The second statement follows the same argument, since if the Firoozbakht's conjecture is true, then given $\varepsilon > 0$ we have $g_n <n^\varepsilon$, for every $n$ sufficiently large. We leave the details to the reader. \end{proof}

This conjecture appears for the first time in \cite{SS}. Other consequence that is easily obtained by the Firoozbakht's conjecture is Andrica's conjecture (see \cite{Andrica}):

\begin{cons} [Andrica]

$g_n < 2\sqrt{p_n} + 1, \ \forall \ n \in \mathbb{N}$.

\end{cons}

\begin{proof} By calculus, we know that $$\ln^2(x) - \ln(x) - 1 < 2\sqrt{x} + 1, \ \forall \ x \geq 1.$$ So, if the Firoozbakht's conjecture is true, then $$g_n < \ln^2(p_n) - \ln(p_n) - 1 < 2\sqrt{p_n} + 1, \ \forall \ n \geq 10.$$ Verifying the remaining cases ($1 \leq n \leq 9$), we finish the proof. \end{proof}

Another beautiful consequence of the Firoozbakht's conjecture is the Oppermann's conjecture (see \cite{Oppermann}):

\begin{cons} [Oppermann]

$\pi(n^2 - n) < \pi(n^2) < \pi(n^2+n), \ \forall \ n \in \mathbb{N} - \{1\}$.

\end{cons}

\begin{proof} Oppermann's conjecture is true for $n \in \{2, \ldots, 74\}$. Now, suppose that exists $n \in \mathbb{N}$, $n \geq 75$, such that $\pi(n^2 - n) = \pi(n^2)$, and let $m = \pi(n^2 - n)$. Then $$p_m < n^2 - n < n^2 < p_{m + 1}.$$ Therefore $$g_m = p_{m + 1} - p_m > n,$$ which implies $$n < g_m < \ln^2(p_m) - \ln(p_m) - 1 < \ln^2(p_m) < \ln^2(n^2 - n).$$ But, by calculus, $$x > \ln^2(x^2 + x), \ \forall \ x \geq 75.$$ This is an absurd. So $\pi(n^2 - n) < \pi(n^2)$. A similar argument for the other inequality finishes the proof. \end{proof}

A direct consequence of assuming Oppermann's conjecture is (see \cite{Legendre})

\begin{cons} [Legendre, strong form]

There are at least two prime numbers between two consecutive squares.

\end{cons}

This, in turn (as $p_{n+1}  \geq p_n + 2, \ \forall \ n \in \mathbb{N}-\{1\}$), implies (see \cite{Brocard})

\begin{cons} [Brocard]

There are at least four prime numbers between $p_n^2$ and $p_{n + 1}^2, \ \forall \ n \in \mathbb{N} - \{1\}$.

\end{cons}

To finish this section, we invite the reader to prove another consequence of the Firoozbakht's conjecture:

\begin{cons} [Legendre's conjecture for cubes, strong form]

There are at least four prime numbers between two consecutive cubes.

\end{cons}

\begin{proof} Left to the reader. \end{proof}

\section{Related conjectures}

Recently, three other conjectures were made related to the Firoozbakht's conjecture. For the first two we refer the reader to \cite{OEIS}. The third conjecture was communicated directly to the first author by a friendly e-mail.

\begin{con} [Nicholson, 2013]

$\left(\displaystyle \frac{p_{n + 1}}{p_n}\right)^n < n \log(n), \ \forall \ n \geq 5$.

\end{con}

\begin{con} [Forgues, 2014]

$\left(\displaystyle \frac{\log(p_{n + 1})}{\log(p_n)}\right)^n < e, \ \forall \ n \in \mathbb{N}$.

\end{con}

\begin{con} [Farhadian, 2016]

$p_n^{\ \left[(\frac{p_{n+1}}{p_n})^n\right]} \leq n^{p_n}, \ \forall \ n \geq 5$.

\end{con}

To see how this conjectures are related to Firoozbakht's conjecture, we will remember the following inequalities:

\begin{thm}

$p_n > n\ln(n), \ \forall \ n \in \mathbb{N}$; and $$\ln(n) + \ln(\ln(n)) - 1 < \displaystyle \frac{p_n}{n} < \ln(n) + \ln(\ln(n)), \ \forall \ n \geq 6.$$

\end{thm}

For a proof of the first inequality (Rosser's theorem), we refer the reader to \cite{Rosser}. For the second, to \cite{Dusart}.

\begin{thm}

Farhadian $\Rightarrow$ Nicholson $\Rightarrow$ Firoozbakht $\Rightarrow$ Forgues.

\end{thm}

\begin{proof} Suppose Farhadian's conjecture is true. Then taking logarithms, $$\left(\displaystyle \frac{p_{n+1}}{p_n}\right)^n \leq \displaystyle \frac{p_n\ln(n)}{\ln(p_n)}.$$ By the last theorem, $$\displaystyle \frac{p_n}{n} < \ln(n) + \ln(\ln(n)) = \ln(n\ln(n)) < \ln(p_n), \ \forall \ n \geq 6,$$ which implies Nicholson's conjecture. By Rosser's theorem, $$\left(\displaystyle \frac{p_{n + 1}}{p_n}\right)^n < p_n, \ \forall \ n \in \mathbb{N},$$ which is equivalent to $$p_{n+1}^n < p_n^{n+1}, \ \forall \ n \in \mathbb{N},$$ which is equivalent to Firoozbakht's conjecture. Now, taking logarithms on the last inequality, we obtain $$\displaystyle \frac{\log(p_{n + 1})}{\log(p_n)} < 1 + \displaystyle \frac{1}{n}.$$ Raising to the $n$-power, we get $$\left(\displaystyle \frac{\log(p_{n + 1})}{\log(p_n)}\right)^n < \left(1 + \displaystyle \frac{1}{n}\right)^n < e, \ \forall \ n \in \mathbb{N},$$ which is the Forgues's conjecture. \end{proof}

\section{A consequence of the Zhang's theorem}

Zhang's theorem is the following celebrated statement.

\begin{thm} [Zhang]

$\liminf g_n < \infty$.

\end{thm}

For a proof of the theorem we refer the reader to \cite{Zhang}. One of the consequences of the Zhang's theorem is that the Firoozbakht's conjecture is true for infinitely many values of $n$.

\begin{thm}

There are infinitely many $n \in \mathbb{N}$ such that $\sqrt[n]{p_n} > \sqrt[n+1]{p_{n+1}}$.

\end{thm}

\begin{proof} (Following \cite{Luan}) Let $$z = \liminf p_{n+1} - p_n$$ and $$Z = \{n \in \mathbb{N}; \ p_{n+1} - p_n = z\}.$$ Then $Z$ is infinite and $$\displaystyle \frac{\fantasma{p}{n}{n+1}}{\fantasma{p}{n+1}{n}} = \displaystyle \frac{\fantasma{p}{n}{n+1}}{(p_n + z)^n} = \left(\displaystyle \frac{p_n}{p_n + z}\right)^n \cdot p_n = \left[\left(\displaystyle \frac{1}{1+\frac{z}{p_n}}\right)^\frac{p_n}{z}\right]^\frac{zn}{p_n} \cdot p_n, \ \forall \ n \in Z.$$ As $$\displaystyle \lim_{x \rightarrow \infty} \left(\displaystyle \frac{1}{1+\frac{1}{x}}\right)^x = \displaystyle \frac{1}{e}$$ and $$\displaystyle \lim_{n \rightarrow \infty} \displaystyle \frac{zn}{p_n} = 0$$ (by Rosser's theorem), then $$\left[\left(\displaystyle \frac{1}{1+\frac{z}{p_n}}\right)^\frac{p_n}{z}\right]^\frac{zn}{p_n} \rightarrow 1 \text{ if } n \rightarrow \infty, \ n \in Z.$$ So $$\displaystyle \lim_{{n \rightarrow \infty} \atop {n \in Z}} \displaystyle \frac{\fantasma{p}{n}{n+1}}{\fantasma{p}{n+1}{n}} = \infty.$$ In particular, $\fantasma{p}{n}{n+1} > \fantasma{p}{n+1}{n}$, for all sufficiently large $n \in Z$. This is equivalent to the statement of the theorem. \end{proof}

\section{Final Remarks}

It is worth commenting on the current state of the relationship of the Firoozbakht's conjecture with two other rather famous problems in number theory: the Riemann hypothesis and the Cramér's conjecture.


Cramér showed that the Riemann hypothesis implies $g_n \ll \sqrt{p_n}\ln(p_n)$. For a proof of this result we refer the reader to \cite{Cramer1} or \cite{Cramer2}. On the other hand, Firoozbakht's conjecture implies that the function $g_n$ has an order of growth  $\ln^2(p_n)$, which is much smaller than the previous result, derived from one of today's most celebrated mathematical problems. However, the authors do not know if the Riemann hypothesis implies that the function $g_n$ grows strictly slower than $\sqrt{p_n}\ln(p_n)$.

The statement of the Cramér's conjecture is $g_n \ll \ln^2(p_n)$. It is clear that Firoozbakht's conjecture implies Cramér's conjecture, according to theorem \ref{firoozbakhtlnquadrado}. Cramér presented his conjecture based on a probabilistic model of prime numbers. In this model, Cramér showed that the relation $$\displaystyle \limsup_{n \rightarrow \infty} \displaystyle \frac{g_n}{\ln^2(p_n)} = 1$$ is true with probability $1$ (see \cite{Cramer2}). However, as pointed out by A. Granville in \cite{Granville}, Maier's theorem shows that the Cramér's model does not adequately describe the distribution of primes numbers on short intervals, and a refinement of the Cramér's model taking into account divisibility by small primes suggests that $$\displaystyle \limsup_{n \rightarrow \infty} \displaystyle \frac{g_n}{\ln^2(p_n)} \geq 2e^{-\gamma} \approx 1,1229...,$$ where $$\gamma = \lim_{n \rightarrow \infty} \left(-\ln(n) + \displaystyle \sum_{k=1}^{n} \displaystyle \frac{1}{k}\right) \approx 0,5772...$$ is the Euler-Mascheroni constant. More specifically, if $\pi(x)$ is the prime-counting function, then the Cramér's model predicts that $$\displaystyle \lim_{n \rightarrow \infty} \displaystyle \frac{\pi(x+(\ln(x))^\lambda) - \pi(x)}{(\ln(x))^{\lambda - 1}} = 1, \ \forall \ \lambda \geq 2,$$ while Maier's theorem states that if $\lambda > 1$, then $$\liminf_{n \rightarrow \infty} \displaystyle \frac{\pi(x+(\ln(x))^\lambda) - \pi(x)}{(\ln(x))^{\lambda - 1}} < 1$$ and $$\limsup_{n \rightarrow \infty} \displaystyle \frac{\pi(x+(\ln(x))^\lambda) - \pi(x)}{(\ln(x))^{\lambda - 1}} > 1.$$

For a proof of Maier's theorem we refer the reader to \cite{Maier}. All this shows how important is a depth study of the Firoozbakht's conjecture, or, more generally, of the function $g_n$.

\section*{Acknowledgments}


The first author would like to thank CAPES by the financial support, to his advisor Hugo Luiz Mariano for all the support given during the last four years, to professor Paolo Piccione and to IME - USP. Both authors are grateful to  professor Marcos Martins Alexandrino da Silva for his valuable suggestions.

\end{document}